\documentclass[10pt,a4paper]{article}
\usepackage[utf8]{inputenc}
\usepackage[T1]{fontenc}
\usepackage{amsmath,amssymb,amsthm,mathtools}
\usepackage{csquotes}
\usepackage{geometry}
\usepackage{bm}
\usepackage{hyperref}
\usepackage{algorithm}        
\usepackage{algpseudocode}    
\floatname{algorithm}{Scheme} 
\usepackage{placeins} 
\counterwithin{algorithm}{section}
\usepackage{enumitem}
\usepackage{subcaption}
\usepackage[T1]{fontenc}   
\usepackage[utf8]{inputenc}
\usepackage{afterpage}
\usepackage[T1]{fontenc}
\usepackage{microtype}
\usepackage{caption}
\usepackage{noto}
\numberwithin{equation}{section}
\theoremstyle{plain}
\newtheorem{theorem}{Theorem}[section]
\newtheorem{lemma}[theorem]{Lemma}

\theoremstyle{definition}

\newtheorem{remark}{Remark}[section]

\newcommand{\E}{\mathbb{E}}

\newcommand{\norm}[1]{\lVert #1 \rVert}

\title{Higher order numerical schemes for SPDEs with {\em additive} Noise}
\author{Abhishek Chaudhary, Andreas Prohl}
\begin{document}
\maketitle
\medskip
	\centerline{Mathematisches Institut Universität Tübingen, Auf der Morgenstelle 10, 72076 Tübingen, Germany}\centerline{chaudhary@na.uni-tuebingen.de, prohl@na.uni-tuebingen.de}
\begin{abstract}
We present high-order numerical schemes for linear stochastic heat and wave equations with Dirichlet boundary conditions, driven by {\em additive} noise. Standard Euler schemes for SPDEs are limited to an order convergence between $1/2$ and $1$ due to the low temporal regularity of noise. For the stochastic heat equation, a {\em modified Crank-Nicolson scheme} with proper numerical quadrature rule for the noise term in its reformulation as {\em random} PDE achieves a strong convergence rate of $3/2$. For the stochastic wave equation with {\em additive} noise a corresponding approach leads to a scheme which is of order 2.
\end{abstract} 
\noindent
	
\section{Introduction}
\label{sec:intro}

Stochastic partial differential equations (SPDEs) provide robust mathematical frameworks for modeling complex systems influenced by {\em randomness} across fields such as physics, engineering, finance, and biology. These equations describe phenomena like heat diffusion under {\em random} forcing, wave propagation with stochastic perturbations, financial models with spatial correlations, and biological processes affected by environmental noise. In this paper, we focus on two fundamental SPDEs: the linear parabolic stochastic heat equation and the stochastic wave equation, both driven by {\em additive} noise and defined on a smooth bounded domain $D \subset \mathbb{R}^d$ with Dirichlet boundary conditions.

A fundamental challenge in the numerical analysis of SPDEs stems from the limited temporal regularity of the driving Wiener process which is only Hölder continuous of order at most $1/2$ in time. This inherent roughness directly constrains the strong convergence rates attainable by a standard numerical analysis of temporal discretizations --- as is shortly evidenced in Remark~\ref{remark1.1}. For instance, typical error analyses for the classical Euler--Maruyama scheme for parabolic SPDEs typically lead to a strong order of convergence of at most $1/2$. An improvement is obtained in \cite{BreitProhl2023IMA, BreitProhlAllenCahn2024}, where an order of convergence $1$ is established for an implicit Euler discretization of SPDEs with colored \emph{additive} noise. This result is achieved by employing a transformation technique that reformulates the original SPDE into a \emph{random} PDE. In this work, we use this idea as an {\em algorithmic} tool in first line to  develop high-order numerical schemes for both, the stochastic heat and wave equations driven by additive noise.

\subsection{Stochastic Heat Equation}

Let $D\subset\mathbb{R}^d$ be smooth, bounded domain, and $T>0$. We consider the stochastic heat equation driven by {\em additive} noise:
\begin{equation}\label{eq:heat}
\begin{cases}
\mathrm{d}X(t)+A X(t)\,\mathrm{d}t =\Phi\,\mathrm{d}W(t), &\text{in}\,(0,T]\times D,\\[4pt]
X=0, & \text{on}\,(0,T]\times \partial D,\\[4pt]
X(0)=X_0, & \text{in}\,D,
\end{cases}
\end{equation}
where $W$ is an $\mathbb{R}^m$-valued Wiener process, the initial datum $X_0:D\to\mathbb{R}$, $\Phi:D\to\mathbb{R}^m$ is the {\em additive} noise coefficient, and $A:\mathbb{H}^2(D)\cap\mathbb{H}_0^1(D)\to\mathbb{L}^2(D)$ is a second-order linear uniformly elliptic operator of the form 
\begin{align}\label{operator} A=\sum_{i,j=1}^{d}a_{i,j}(x)\partial_{x_i,x_j}+\sum_{i=1}^d b_i(x)\partial_{x_i}+ c(x),
\end{align}
with smooth enough coefficients $a_{i,j}, b_i, c$ for all $i,j=1,\dots,d$\,; see, {\em e.g.} \cite[Sec. 3.1]{Lunardi1995}. A standard approach to time-discretization of SPDE~\eqref{eq:heat} is the implicit Euler-Maruyama scheme:
\begin{equation}\label{eq:implicit_euler_discrete}
\begin{cases}
\displaystyle X_{j+1}-X_j+ \tau AX_{j+1}= \Phi\Delta_{j+1}W, & j=0,\dots,N-1,\\[4pt]
X(0)=X_0,\\[4pt]
X_j\big|_{\partial D}=0, \qquad j=1,\dots,N,
\end{cases}
\end{equation}
with a uniform time grid $t_j = j\tau$, $j=0,1,\dots,N$, $\tau = \frac{T}{N}$ and the Wiener increment $\Delta_{j+1}W=W(t_{j+1})-W(t_{j})$. Due to the low temporal regularity of the noise, typical proofs of strong convergence for \eqref{eq:implicit_euler_discrete} show a rate of $1/2$, see {\em e.g.} \cite{Yan2005}, a limited order result that is shortly evidenced in the following remark.  
\begin{remark}\label{remark1.1}
Let \( e_j = X(t_j) - X_j \), where $\{X_j;1\le j\le N\}$ solves the scheme~\eqref{eq:implicit_euler_discrete}. In integrated form on $[t_j, t_{j+1}]$, the first equation in ~\eqref{eq:heat} rea{d}s as
\[
X(t_{j+1}) - X(t_j) + \int_{t_j}^{t_{j+1}} A X(s) \,{\rm d}s = \Phi \Delta_{j+1}W.
\]
Taking differences of this equation and (\ref{eq:implicit_euler_discrete}) amounts to
\begin{align}\label{today05}
    e_{j+1} - e_{j} + \int_{t_j}^{t_{j+1}} A \bigl[ X(s) -X_{j+1}\bigr] \, {\rm d}s = 0.
\end{align}
The order limiting term is now the integrand, since the solution $\{X(t); t\in[0,T]\}$ to \eqref{eq:heat} is known to be H\"older continuous in time up to $1/2$.
\end{remark}
Different ansatzes to ours given below were considered in the last decades to construct a scheme of (improved) strong order $1$, which employs the {\em explicit} knowledge of the related semigroup for \eqref{operator} to approximately solve \eqref{eq:heat}. A particular example where eigenvalues and -functions  are explicitly known is $A = -\Delta$ on the torus to use the related semigroup for an algorithm; see {\em e.g.} \cite{D1} and the cited references there, as well as how to herewith set up  the exponential Euler method in this context to settle strong order $1$. Despite its elegant analysis, such algorithms are of limited use in practice since their related characterizations are not available for {\em general} $A$ in \eqref{operator}, acting on {\em general} domains $D$. We therefore here aim to construct implementable higher-order methods that may be
used to approximate the solution of \eqref{eq:heat}  for general data settings — {\em i.e.}, general $A$ from \eqref{operator}, as well as general  smoothly bounded domains $D$.

\noindent
To accomplish this goal, we take motivation from the works \cite{BreitProhl2023IMA, BreitProhlAllenCahn2024}, and consider instead of the solution $X$ from \eqref{eq:heat} the transformation 
\begin{align}\label{z}
    u = X - \Phi W,
\end{align}and therefore convert the SPDE \eqref{eq:heat} into a {\em random} PDE:
\begin{equation}\label{eq:heat-uPDE}
\begin{cases}
u_t(t)+A u(t)=-A\bigl[\Phi W(t)\bigr], &\text{in}\,(0,T]\times D,\\[4pt]
u=0, &\text{on}\,(0,T]\times \partial D,\\[4pt]
u(0)=X_0, & \text{in}\,D,
\end{cases}
\end{equation}
where $u_t:=\partial_t u$. It is known that the solution $u$ of \eqref{eq:heat-uPDE} now {\em does} possess temporal derivatives; in fact, it can be shown that $u_t\in \mathbb{L}_{\mathbb{F}}^2(\Omega;\mathbb{L}^2([0,T];\mathbb{L}^2(D))$, provided $X_0\in \mathbb{H}_0^1(D)$ and $\Phi\in (\mathbb{H}^2(D)\cap \mathbb{H}_0^1(D))^m$; see \eqref{today03}. This is the guiding motivation to verify order $1$ for \eqref{eq:implicit_euler_discrete} in the context of more regular data $X_0\in\mathbb{H}_0^1(D)\cap\mathbb{H}^2(D)$ and $\Phi\in(\mathbb{H}_0^1(D)\cap\mathbb{H}^2(D))^m$ with \ compatibility condition $A\Phi\in (\mathbb{H}_0^1(D))^m$. To verify this result in \cite{BreitProhl2023IMA, BreitProhlAllenCahn2024} — and instead of starting with \eqref{eq:heat} directly --- we consider first implicit Euler scheme for \eqref{eq:heat-uPDE} which reads as
\begin{equation}\label{eq:heat-uCN}
\begin{cases}
\displaystyle U_{j+1} - U_j + \tau A U_{j+1}=-\tau A \bigl[\Phi W(t_{j+1})\bigr], & j=0,\dots,N-1,\\[8pt]
U^0 = X_0,\\[4pt]
U_j\big|_{\partial D}=0, \qquad j=1,\dots,N,
\end{cases}
\end{equation}
and exploit improved the temporal regularity for the solution $u$ of \eqref{eq:heat-uPDE} to verify
\begin{align}\label{eq:heat-error}
\sup_{1\leq j\leq N} \mathbb{E}\left[\|u(t_j) - U_j\|_{\mathbb{L}^2(D)}^2\right] \leq C\tau^2.
\end{align}
Afterwards, we revert the transformation procedure on a discrete level, by defining $X^j:=U_j+\Phi W(t_j)$ as approximate for $X(t_j)$ from \eqref{eq:heat}. In fact, it can be immediately seen that $X^j=X_j$, such that we have
\begin{align}\label{0001}
\sup_{1\leq j\leq N} \mathbb{E}\left[\|X(t_j) - X_j\|_{\mathbb{L}^2(D)}^2\right]= \sup_{1\leq j\leq N} \mathbb{E}\bigl[ \Vert u(t_j) - U_j\Vert^2_{\mathbb{L}^2(D)}\bigr] \leq C\tau^2 .
\end{align}
This convergence proof of rate $1$ for implicit scheme $\eqref{eq:implicit_euler_discrete}$ via \eqref{eq:heat-uPDE} --and \eqref{eq:heat-uCN} -- in the presence of regular data $X_0$ and $\Phi$ nourishes the hope to construct even higher order methods since the time regularity of the solution $u$ to {\em random} PDE \eqref{eq:heat-uPDE} is even higher, {\em i.e.},
\begin{align}\label{012}
u \in C^{1,{1}/{2}}\big([0,T]; \mathbb{L}^2(\Omega; \mathbb{H}_0^1(D))\big),
\end{align}
if data $X_0\in \mathbb{H}^4(D)\cap\mathbb{H}_0^1(D)$ and $\Phi\in \mathbb{H}^3(D)\cap\mathbb{H}_0^1(D)$ with compatibility conditions $AX_0\in \mathbb{H}_0^1(D)$ and $A\Phi\in(\mathbb{H}_0^1(D))^m$; see \eqref{eq:heat-regularity}. This result for the solution $u$ of \eqref{eq:heat-uPDE} again rests on regularity theory for deterministic PDEs with temporal H\"older regularity of the right-hand side; see \cite[Thm. 4.3.1]{Lunardi1995}. For this purpose, we propose the new {\em modified Crank--Nicolson Scheme}~\ref{eq:heat-schemeCN-SPDE} for SPDE \eqref{eq:heat} which is a strong convergence rate of $3/2$; see Theorem~\ref{thm:main}.
\subsection{Stochastic wave equation}
Let $D\subset\mathbb{R}^d$ be smooth, bounded domain, and $T>0$. We consider the stochastic wave equation driven by {\em additive} noise:
\begin{equation}\label{eq:wave}
\begin{cases}
\displaystyle  {\rm d}X_t(t) + AX(t)\,{\rm d}t = \Phi\,{\rm d}W(t), 
& \text{in}\,(0,T]\times D,\\[6pt]
X(0)=X_0,\quad  X_t(0)=X_1, & \text{in}\,D,\\[6pt]
X(t)=0, &\text{on}\,[0,T]\times\partial D,
\end{cases}
\end{equation}
where $X$ and $X_t:=\partial_t X$ denote displacement and velocity, and $X_0$ and $X_1$ are given initial data; again, $A$ is a second-order linear uniformly elliptic operator from \eqref{operator}. Next to the (explicit) Euler method, first numerical schemes proposed for \eqref{eq:wave} also use the explicit knowledge of the related semigroup for $A = -\Delta$ in $D=(0,1)$ to construct methods of order $1$; see \cite{C1, P1} for surveys of the relevant literature. A different viewpoint takes \cite{P1} where a numerical scheme of strong order $1.5$ is given for the stochastic wave equation driven by a multiplicative noise, which is applicable for general $A$ from \eqref{operator} and smooth, bounded domains $D \subset {\mathbb R}^d$.  To even construct a higher order approximations for SPDE~\ref{eq:wave}, we use another transformation than \eqref{z} here, which is 
\begin{align}\label{today11} 
u(t) = X(t)- \int_0^t\Phi W(s) \, \mathrm{d}s,
\end{align}
to reformulate \eqref{eq:wave} as a {\em random} PDE:
\begin{equation}\label{eq:wave-transformed}
\begin{cases}
\displaystyle u_{tt}(t)+Au(t) =- \displaystyle\int_{0}^{t} A\bigl[\Phi W(s)\bigr]\,\mathrm{d}s,
&\text{in}\,(0,T]\times D,\\[8pt]
u(t,x)=0, & \text{on}\,[0,T]\times\partial D,\\[6pt]
u(0)=X_0,\quad  u_t(0)=X_1, &\text{in}\,D.
\end{cases}
\end{equation}
Similar to \eqref{012}, the solution $u$ to {\em random} PDE~\eqref{eq:wave-transformed} ensures (see~\eqref{eq: wave-regularity}) 
\begin{align}\label{0909} u\in C^{2}([0,T];\mathbb{L}^2(\Omega;\mathbb{H}^2(D)),\qquad\text{and}\qquad u_t\in C^2([0,T];\mathbb{L}^2(\Omega; \mathbb{H}^1(D))).
\end{align}
If compared to \eqref{012}, we observe improved temporal regularity of the solution of $\eqref{eq:wave-transformed}$, which is why we may hope for an improved rate of convergence for the new {\em modified Crank--Nicolson} Scheme~\ref{alg:wave-modified-detailed} whose construction is detailed in Section~\ref{sec:wave}. In fact we prove convergence with rate $2$ for iterates which eventually approximate the solution $X$ of $\eqref{eq:wave}$; see Theorem~\ref{thm:wave-main-precise}. 
\subsection{Numerical simulations}

To evidence higher convergence rates for the modified Crank--Nicolson (MCN) schemes for the stochastic heat equation \eqref{eq:heat} and the stochastic wave equation \eqref{eq:wave}, we conduct numerical experiments on the domain $D = (0,1)$ with homogeneous Dirichlet boundary conditions; see Figure~\ref{fig:convergence_rates}. The noise coefficient is $\Phi(x) = \sin(2\pi x) +\sin(3 \pi x)$ with $m=1$. Initial conditions are $X(0,x) = \sin(\pi x)$ for both, with $X_t(0,x) = 0$ for the wave equation. Simulations are performed over $[0,T]$ with $T=1$, by  using a uniform spatial grid of $K=40$ points (spatial step $h=1/(K+1)$) and time steps $\tau = T/N$ for $N \in \{4, 8, 16, 32, 64, 128, 256, 512, 1024\}$. Errors are computed with the help of $M=1000$ Monte Carlo realizations, with Wiener paths approximated on a fine grid of $1024\times 1024$ steps. For the stochastic {heat equation}\eqref{eq:heat}, the exact solution is for all $t\in[0,T],$
\begin{align*}
X(t,x) &= e^{-\pi^2 t} \sin(\pi x) + \int_0^t e^{-4\pi^2 (t-s)}\sin(2\pi x)\, {\rm d}W(s) + \int_0^t e^{-9\pi^2 (t-s)}\sin(3\pi x)\,{\rm d}W(s) .
\end{align*}
We define $$\mathrm{Error}=\sqrt{\mathbb{E}\big[\|X(T)-X_{T}\|_{\mathbb{L}^2(D)}^2\big]}.$$
We compare the MCN Scheme~\ref{eq:heat-schemeCN-SPDE} against a Euler--Maruyama (EM) scheme~\eqref{eq:implicit_euler_discrete}. The rates of convergence are approximately $1.5$ for MCN Scheme~\ref{eq:heat-schemeCN-SPDE} and $1$ for EM scheme~\eqref{eq:implicit_euler_discrete}; see Figure~\eqref{fig:heat_convergence}. 

For the stochastic {wave equation} \eqref{eq:wave}, due to the complexity of the exact solution, we use a reference solution with $N_{\text{ref}}=4096$. The rate of convergence is approximately $2$ for the MCN Scheme~\ref{alg:wave-modified-detailed}; see Figure~\eqref{fig:wave_convergence}. Here we define 
$$\mathrm{Error}=\sqrt{\mathbb{E}\big[\|\nabla(X(T)-X_{T})\|_{\mathbb{L}^2(D)}^2\big]}\qquad\text{or}\qquad \sqrt{\mathbb{E}\big[\|X_t(T)-Y_{T}\|_{\mathbb{L}^2(D)}^2\big]}.$$
\noindent
\begin{figure}[htbp]
    \centering
    \begin{subfigure}{0.49\textwidth}
        \centering
        \includegraphics[width=\textwidth]{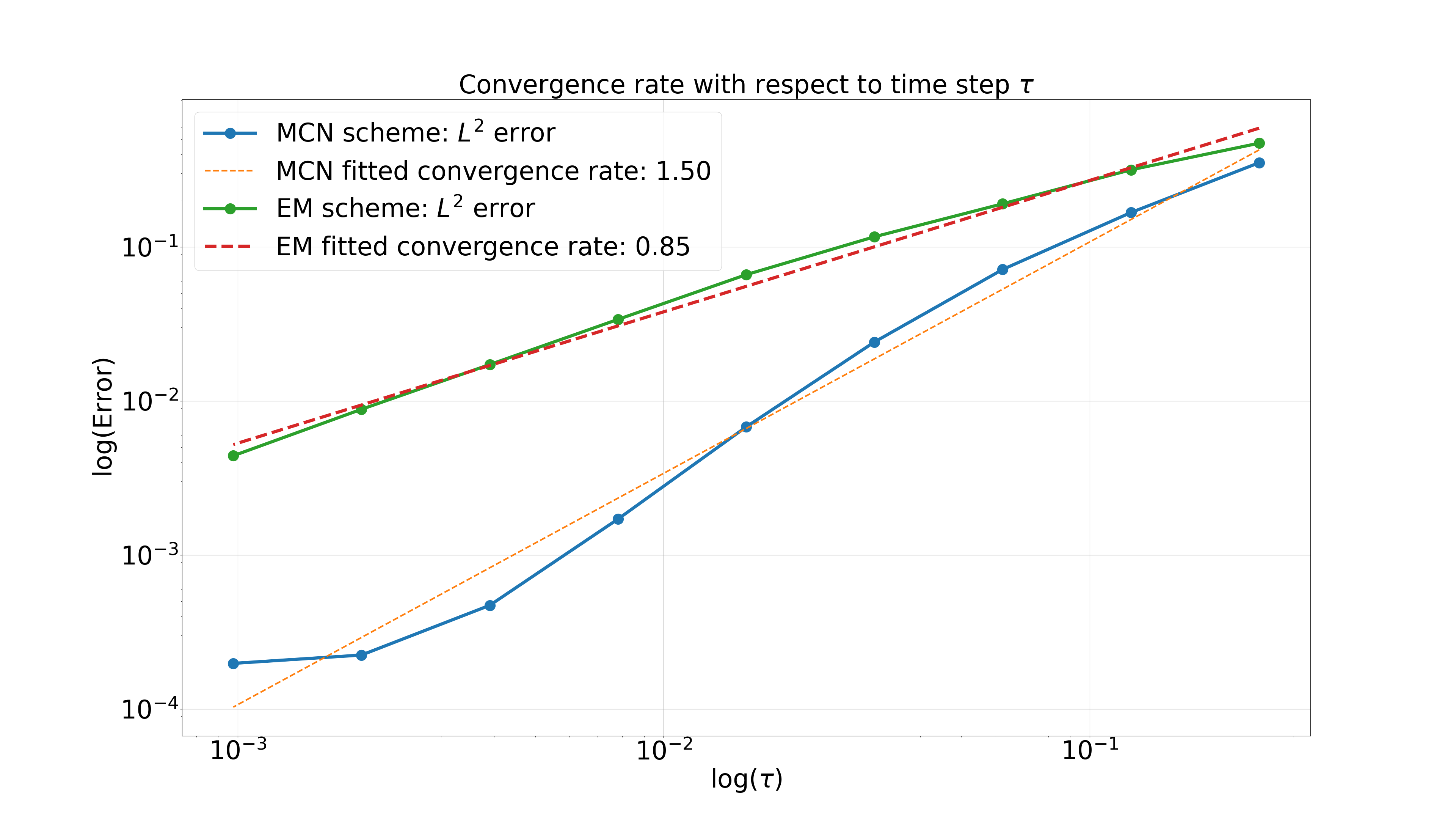}
        \caption{The MCN Scheme~\ref{eq:heat-schemeCN-SPDE} achieves a rate of approximately 1.5, while the EM scheme~\eqref{eq:implicit_euler_discrete} yields a rate of 0.85 for the stochastic heat equation~\eqref{eq:heat}.}
        \label{fig:heat_convergence}
    \end{subfigure}
    \hfill
    \begin{subfigure}{0.49\textwidth}
        \centering
        \includegraphics[width=\textwidth]{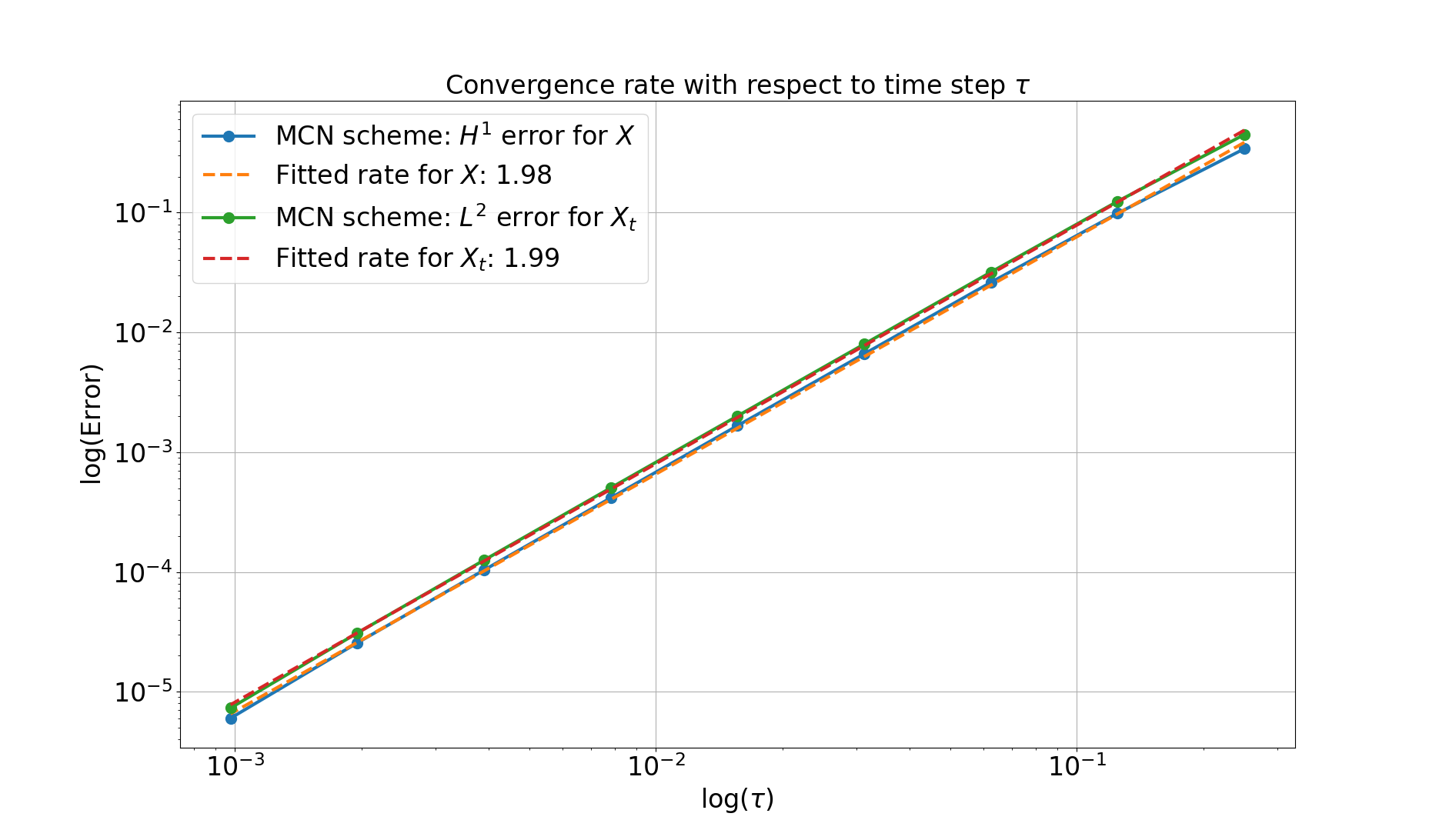}
        \caption{The MCN Scheme~\ref{alg:wave-modified-detailed} achieves a rate of approximately $2$ for the stochastic wave equation~\eqref{eq:wave}.} 
        \label{fig:wave_convergence}
    \end{subfigure}
    \caption{Convergence rates for the stochastic heat equation~\eqref{eq:heat} (left) and wave equation~\eqref{eq:wave} (right), averaged over $1000$ Monte Carlo realizations.  }
    \label{fig:convergence_rates}
\end{figure}




\section{Mathematical setup and preliminarily results}\label{section 2}
Let \(D\subset\mathbb{R}^d\) be a smooth bounded domain. Let \((\Omega, \mathcal{F}, \mathbb{F}, \mathbb{P})\) with $\mathbb{F} = \{\mathcal{F}_t;\,t\in[0,T]\}$ be a complete filtered probability space.
Let $\mathrm{B}$ be a Banach space. Let $$\mathbb{L}^2_{\mathbb{F}}\bigl(\Omega\times[0,T];\mathrm{B}\bigl)\qquad \text{and} \qquad\mathbb{L}^2_{\mathbb{F}}\bigl([0,T];C([0,T];\mathrm{B})\bigr)$$ be the spaces of $\mathrm{B}$-valued $\mathbb{F}$-adapted stochastic processes and $\mathrm{B}$-valued continuous $\mathbb{F}$-adapted stochastic processes, respectively. For some $m\in\mathbb{N}$, let \(W=(W_1,...,W_m) \) be an \( \mathbb{R}^m \)-valued $\mathbb{F}$-adapted Wiener process on the filtered probability space \((\Omega, \mathcal{F},  \mathbb{F},  \mathbb{P})\). Let $\Phi=(\Phi_1,...,\Phi_m)$ be a $(\mathbb{L}^2(D))^m$-valued {\em additive} noise coefficient. We then define 
$$\Phi W(t)=\sum_{i=1}^m\Phi_iW_i(t).$$  
Let $\left <\cdot,\cdot\right >_\mathrm{K}$ denote the inner product on Hilbert space $\mathrm{K}$. Let $<\cdot,\cdot>_{\mathcal{D}}$ denote the dual pairing between $\mathbb{H}^{-1}(D)$ and $\mathbb{H}_0^1(D)$. For convenience, we define 
\[
\langle u,v\rangle := \int_D u(x)\,v(x)\,{\rm d}x.
\]
We also use the standard fact that if $f \in C([0, T]; \mathrm{B})$ and $f' \in L^1([0, T]; \mathrm{B})$, then for every subinterval $[a, a + \kappa] \subset [0, T]$, the fundamental theorem of calculus for Bochner integrals holds with
$$
f(a + s) = f(a) + \int_0^s f'(a + \tau) \, \mathrm{d}\tau, \qquad 0 \leq s \leq \kappa.
$$
The following lemma addresses the quadrature error in the context of the Crank--Nicolson scheme where the integrand is of limited regularity (see, {\em e.g.,}\cite[Theorem 2]{DragomirMabizela2000}).
\begin{lemma}\label{lem:mean-value-fractional-vector}
Let \((\mathrm{B}, \|\cdot\|_{\mathrm{B}})\) be a Banach space. Assume \(f \in C^{1,\gamma}([0,T]; \mathrm{B})\) for some \(\gamma \in (0,1]\)and there exists a constant $\widetilde{C}>0$ such that
\begin{align}\label{today17}
\|f'(t) - f'(s)\|_\mathrm{B} \leq \widetilde{C} |t - s|^\gamma \quad \forall s, t \in [0,T]\,.
\end{align}
Then for any interval \([a, a+\kappa] \subset [0,T]\) we have that
\[
\left\| \frac{f(a) + f(a+\kappa)}{2} - \frac{1}{\kappa} \int_a^{a+\kappa} f(\xi)\,\mathrm{d}\xi \right\|_\mathrm{B} \leq \frac{\widetilde{C}}{(\gamma+2)(\gamma+3)} \kappa^{1+\gamma}.
\]
\end{lemma}
\section{Stochastic Heat Equation}
\label{sec:heat}
Let \(D\subset\mathbb{R}^d\) be a smooth bounded domain. Let $(\Omega, \mathcal{F},\mathbb{F}, {\mathbb P})$ be a given probability space, let $W$ as specified in Section~\ref{section 2} and $A$ from \eqref{operator}. The existence of a unique variational solution $X \in \mathbb{L}^2_{\mathbb{F}}\big(\Omega;{C}([0, T];\mathbb{L}^2(D))\big) \cap \mathbb{L}^2_{\mathbb{F}}\big(\Omega;\mathbb{L}^2([0,T];\mathbb{H}_0^1(D))\big)$ to SPDE \eqref{eq:heat} may {\em e.g.} be found in \cite[Ch. 3]{Chow2014}, which satisfies
 \begin{align}\label{today01}
         \E\bigl [\|X\|^2_{{C}([0,T];\mathbb{L}^2(D))}\bigr] + \E\bigl [\| X \|^2_{\mathbb{L}^2([0,T];\mathbb{H}_0^1(D))}\bigr ]\leq C \big(\|X_0 \|^2_{\mathbb{L}^2(D)} + \|\Phi\|_{(\mathbb{L}^2(D))^m}^2 \big).
    \end{align}
Moreover, by standard regularity theory for more regular data $X_0\in \mathbb{H}_0^1(D)$ and $\Phi\in (\mathbb{L}^2(D))^m$, it satisfies 
    \begin{align}\label{today02}
         \E\bigl [\|X\|^2_{{C}([0, T];\mathbb{H}_0^1(D))}\bigr ] + \E\bigl [\|X \|^2_{\mathbb{L}^2([0, T];\mathbb{H}^2(D))}\bigr ]\leq C \big(\| X_0 \|^2_{\mathbb{H}_0^1(D)} +  \|\Phi\|_{(\mathbb{L}^2(D))^m}^2\big).
    \end{align}
For the {\em random} PDE \eqref{eq:heat-uPDE}, if $X_0\in \mathbb{H}_0^1(D)$ and $\Phi\in(\mathbb{H}^2(D)\cap \mathbb{H}_0^1(D))^m$,  we have the existence and uniqueness of a variational solution $$ u \in \mathbb{L}^2_\mathbb{F}\big(\Omega ; {C}([0,T];\mathbb{H}_0^1(D))\big)\cap\mathbb{L}^2_{\mathbb{F}}\big(\Omega;\mathbb{L}^2([0,T];\mathbb{H}^2(D))\big),$$
such that $u_t\in \mathbb{L}_{\mathbb{F}}^2\big(\Omega;\mathbb{L}^2([0,T];\mathbb{L}^2(D))\big),$  and  the following variational formulation holds: $\mathbb{P}$-a.s., for all $\varphi\in \mathbb{H}_0^1(D)$, a.e. $s\in [0,T]$,
\begin{align*}
    \left< u_{t}(s), \varphi\right>=-\left<\nabla u(s), \nabla \varphi\right>+ \left<\Delta\bigl[\Phi W(s)\bigr],\varphi\right>,
\end{align*}
with $u(0)=X_0$. In addition, the following estimate holds
\begin{align}\label{today03}
\mathbb{E}\bigl [\|u\|_{{C}([0,T];\mathbb{H}_0^1(D))\cap \mathbb{L}^2([0,T];\mathbb{H}^2(D))}^2\bigr ]+\E\bigl [\|u_t\|_{\mathbb{L}^2([0,T];\mathbb{L}^2(D))}^2\bigr ]\le C\big(\|X_0\|_{\mathbb{H}_0^1(D)}^2+ \|\Phi\|_{(\mathbb{H}^2(D))^m}^2\big)\,;
\end{align}
 see, {\em e.g.,} \cite[Ch. 7, Sec. 1]{Evans2010}. Moreover, if $X_0\in \mathbb{H}^4(D)\cap \mathbb{H}_0^1(D)$ and $\Phi\in(\mathbb{H}^3(D)\cap\mathbb{H}_0^1(D))^m$ with compatibility condition $AX_0, A\Phi\in\bigl(\mathbb{H}_0^1(D)\bigr)^m$, then with the help of H\"older continuity of $A[\Phi W]$ in time \big({\em i.e.,} $A [\Phi W]\in C^{1/2}\bigl([0,T];\mathbb{L}^2(\Omega;\mathbb{H}_0^1(D))\bigr)$\big), the solution $u$ to {\em random} PDE \eqref{eq:heat-uPDE} satisfies $u \in C^{1,1/2}\big([0, T];\mathbb{L}^2(\Omega; \mathbb{H}_0^1(D))\big),$ with the following higher regularity bound
\begin{align}\label{eq:heat-regularity}
\|u\|_{{C}^{1,1/2}([0,T];\mathbb{L}^2(\Omega;\mathbb{H}_0^1(D)))}^2\le C\big(\|X_0\|_{\mathbb{H}^4(D)}^2+ \|\Phi\|_{(\mathbb{H}^3(D))^m}^2\big).
\end{align}
This result may {\em e.g.} be deduced from \cite[Prop. 4.1.2, Thm~4.3.1]{Lunardi1995} by exploiting temporal H\"older regularity of the Wiener process on the right-hand side of \eqref{eq:heat-uPDE}.
\subsection{Modified Crank--Nicolson Scheme}
For simplicity, in this section, we take $A$ to be the negative Laplace operator. Our main result (Theorem~\ref{thm:main} below), however, remains valid for any $A$ from \eqref{operator}. The strong solution $u$ to the {\em random} PDE \eqref{eq:heat-uPDE} then satisfies the integral identity
\begin{align}\label{today02}
 u(t_{j+1}) - u(t_j)- \int_{t_{j}}^{t_{j+1}}\Delta u(s)  \,\mathrm{d}s = \int_{t_j}^{t_{j+1}} \Delta[\Phi W(s)]  \,\mathrm{d}s.
\end{align}
To exploit the improved time-regularity \eqref{eq:heat-regularity} of $u_t$, we use the Crank-Nicolson quadrature rule for the integral $\int_{t_j}^{t_{j+1}} \Delta u(s)\, {\rm d}s$ in \eqref{eq:heat-uPDE} since its integrand is
Lipschitz in time, and a Riemann sum for the non-Lipschitz integrand in
$\int_{t_j}^{t_{j+1}} \Delta [\Phi W(s)]\, {\rm d}s$ on a finer mesh of size $\tau^2$ to partition intervals $[t_j, t_{j+1}]$. We proceed as follows:
\begin{enumerate}
    \item \textbf{Approximation of the deterministic integral:}
    by using the trapezoidal rule, the term 
   $ 
    \mathcal{Q}_j^\Delta := \int_{t_j}^{t_{j+1}} \Delta u(s)  \,\mathrm{d}s
    $
    is approximated by $\frac{\tau}{2} \Delta \bigl [u(t_{j+1}) + u(t_j) \bigr ] =: \mathcal{I}_j^\Delta.$
    \item \textbf{Approximation of the integral with stochastic integrand:}
    on each interval $[t_j, t_{j+1}]$, we introduce a finer uniform mesh of step size $\tau^2$, with nodes
    \[
    t_{j,\ell} = t_j + \ell \tau^2, \quad \ell = 0, \dots, \mathrm{M},
    \]
    where $\mathrm{M} = \tau^{-1}$. By using a Riemann sum on a finer mesh, the integral
   $
    \mathcal{Q}_j^W := \int_{t_j}^{t_{j+1}} W(s)\,\mathrm{d}s
 $
    is then approximated by 
    $
    \sum_{\ell=1}^{\mathrm{M}} \tau^2  W(t_{j,\ell}) =: \mathcal{I}_j^W.
    $ This idea is inspired from \cite{P1}.
\end{enumerate}
We remark that to apply the trapezoidal rule to $\mathcal{Q}_j^W$ is not successful since its integrand is not Lipshitz, which avoids to benefit from it in terms of accuracy; see requirement~\eqref{today17} in Lemma~\ref{lem:mean-value-fractional-vector}. This is the reason why we use ${\mathcal I}_j^W$ which uses the basic quadrature rule — but on a finer mesh to compensate for its worser approximation property. The resulting {\em modified Crank--Nicolson scheme} for the {\em random} PDE \eqref{eq:heat-uPDE} then is
\begin{equation}
\label{eq:heat-schemeCN}
\begin{cases}
\displaystyle U_{j+1} - U_j -\frac{\tau}{2} \Delta [U_{j+1} + U_j] =\Delta\bigg[\Phi  \sum_{\ell=1}^{\mathrm{M}} \tau^2 W(t_{j,\ell})\bigg], & j=0,\dots,N-1,\\[4pt]
U_0=X_0,\\[4pt]
U_j\big|_{\partial D}=0, \qquad j=1,\dots,N.
\end{cases}
\end{equation}
We remark that the term on the right-hand side of $\eqref{eq:heat-schemeCN}_1$ does not significantly increase the complexity of a standard Crank-Nicolson scheme per path since its computation does {\em not} involve the unknown $U_{j+1}$ and hence can be computed \lq offline\rq.
In order to now have a finite sequence of approximates $\{X_j;\, 0 \leq j \le N\}$  for the solution $X$ to \eqref{eq:heat}, we define $X_j: = U_j + \Phi W(t_j)$, for $0 \leq j \leq N$. We herewith arrive at the new {\em modified Crank--Nicolson scheme} for the SPDE \eqref{eq:heat} -- which reads as follows.
\begin{algorithm}[H]
\caption{{\em modified Crank--Nicolson scheme} for SPDE~\eqref{eq:heat} with $A=-\Delta$}
\label{eq:heat-schemeCN-SPDE}
\begin{algorithmic}[1]
\State \textbf{Inputs:} domain $D$, noise coefficient $\Phi$, final time $T$, total steps $N$, time step $\tau=T/N$, initial condition $X_0$, Wiener process $W$.
\State \textbf{Time grid:} $t_j = j\tau$ for $j=0,\dots,N$.
\State \textbf{Fine grid per interval:} on $[t_j,t_{j+1}]$, define $t_{j,\ell} = t_j + \ell \tau^2$, $\ell=0,\dots,\mathrm{M}$, with $\mathrm{M}=\tau^{-1}$.
\State \textbf{Initialization:} set $X_0\in \mathbb{H}_0^1(D)$. 
\For{$j = 0$ \textbf{to} $N-1$}
\State Compute the Brownian increment: $\Delta_{j+1}W = W(t_{j+1}) - W(t_j)$.
\State Compute the correction term:
    $$
    \mathfrak{E}_{j+1}^\Delta(W) =
    \Delta\!\left[\Phi \sum_{\ell=1}^{\mathrm{M}} \tau^2\, W(t_{j,\ell})\right]
    - \frac{\tau}{2}\,\Delta\!\left[\Phi\big(W(t_{j+1}) + W(t_j)\big)\right].
    $$
\State Compute $X_{j+1}\in H_0(D)$ from the following scheme:
    $$
    X_{j+1} - X_j - \frac{\tau}{2}\,\Delta\,(X_{j+1}+X_j)
    \;=\; \Phi\,\Delta_{j+1}W + \mathfrak{E}_{j+1}^\Delta(W).
    $$
\EndFor
\State \textbf{Output:} $\{X_j;\,0\le j\le N\}$.
\end{algorithmic}
\end{algorithm}

Note that the correction term $\mathfrak{E}_{j+1}^\Delta( W)$ in the {\em modified Crank--Nicolson} Scheme~\ref{eq:heat-schemeCN-SPDE} plays a key role in improving the overall convergence rate. The following theorem settles ${\mathcal O}(\tau^{3/2})$ convergence rate for the {\em modified Crank--Nicolson} Scheme~\ref{eq:heat-schemeCN-SPDE}.

\begin{theorem}[first main result]
\label{thm:main}
Let $X_0\in\mathbb{H}^4(D)\cap\mathbb{H}_0^1(D), \Phi\in\bigl(\mathbb{H}^3(D)\cap\mathbb{H}_0^1(D)\bigr)^m$ with compatibility condition $AX_0\in \mathbb{H}_0^1(D), A\Phi\in\bigl(\mathbb{H}_0^1(D)\bigr)^m$. Let $\{X_j:1\le j\le N\}$ be obtained from the {\em modified Crank-Nicolson} Scheme \ref{eq:heat-schemeCN-SPDE} and $X$ be the unique solution to SPDE~\eqref{eq:heat} with $A=-\Delta$. Then there exists $C\equiv C(T, X_0, \Phi, m)>0$ such that
\[
\sup_{1\le j\le N} \E \bigl[ \norm{X(t_j) - X_j}_{\mathbb{L}^2(D)}^2 \bigr]^{1/2} + \left( \tau \sum_{j=0}^{N-1} \E\left[ \norm{\bigr(X(t_j)+X(t_{j+1})\big) - 
\big(X_j+X_{j+1}\bigl)}_{\mathbb{H}^1_0(D)}^2 \right] \right)^{1/2} \le C\tau^{3/2}.
\]
\end{theorem}
\begin{proof}
We prove this result with the help of the {\em modified Crank--Nicolson scheme}~\eqref{eq:heat-schemeCN}  applied to the \emph{random} PDE~\eqref{eq:heat-uPDE}. The proof proceeds through the following steps.

\noindent
\textbf{Step 1. Error inequality:} The exact solution $u$ to \eqref{eq:heat-uPDE} satisfies the relation
\[
{u(t_{j+1}) - u(t_j)} - \frac{\tau}{2} \Delta \bigl[u(t_{j+1}) + u(t_j)\bigr] = \int_{t_j}^{t_{j+1}} \Delta\bigl [\Phi W(s)\bigr ] \, \mathrm{d}s - \tau \Delta \mathcal{Q}^{u}_{j+1/2},
\]
where $\mathcal{Q}^{u}_{j+1/2}: = \frac{1}{\tau} \int_{t_j}^{t_{j+1}} u(s) \, \mathrm{d}s - \frac{1}{2} \bigl[u(t_{j+1}) + u(t_j)\bigr]$ represents the trapezoidal rule truncation error. 
We define the numerical error $$e_j := u(t_j) - U_j=X(t_j)-X_j.$$ By subtracting the scheme \eqref{eq:heat-schemeCN}, we obtain
\begin{equation*}
{e_{j+1} - e_j} - \frac{\tau}{2} \Delta \bigl[e_{j+1} + e_j\bigr] = \Delta \bigl [\Phi\big(\mathcal{Q}_j^W - \mathcal{I}_j^W\big)\bigr ]-\tau\Delta\mathcal{Q}^{u}_{j+ 1/2}.
\end{equation*}
By taking the $\mathbb{L}^2(D)$-inner product with ${e_{j+1/2}}:= \frac{1}{2}\bigl[e_{j+1} + e_j\bigr]$ yields 
\begin{equation}\label{today07}
 \frac{1}{2}\bigl(\norm{e_{j+1}}_{\mathbb{L}^2(D)}^2 - \norm{e_j}_{\mathbb{L}^2(D)}^2 \bigr) + \frac{\tau}{2} \norm{\nabla e_{j+1/2}}^2 = -\tau \left< \frac{1}{\tau} \nabla \bigl [\Phi\big(\mathcal{Q}_j^W - \mathcal{I}_j^W\big)\bigr ] -\nabla\mathcal{Q}^{u}_{j+1/2}, \nabla{e_{j+1/2}}\right>.
\end{equation}
For any $1\le J\le N$, by applying Young's inequality and taking expectations, and summing over $j = 0$ to $J-1$, we obtain the error inequality
\begin{align}\label{today08}
\mathbb{E}[\norm{e_J}_{\mathbb{L}^2(D)}^2] + \frac{\tau}{2} \sum_{j=0}^{J-1} \mathbb{E}[\norm{\nabla e_{j+1/2}}_{\mathbb{L}^2(D)}^2] \le &\mathbb{E}[\norm{e_0}_{\mathbb{L}^2(D)}^2]+ \frac{C}{\tau}\sum_{j=0}^{J-1} \mathbb{E}\bigl [\norm{\big(\mathcal{Q}_j^W - \mathcal{I}_j^W\big)}_{\mathbb{R}^m}^2\bigr ] \norm{\nabla \Phi}_{(\mathbb{L}^2(D))^m}^2\notag\\&\qquad  + C\tau \sum_{j=0}^{J-1} \mathbb{E}[\norm{\nabla \mathcal{Q}^{u}_{j+1/2}}_{\mathbb{L}^2(D)}^2].
\end{align}
In next two steps, we estimate the right hand side of the error inequality~\eqref{today08}.

\noindent
\textbf{Step 2. Trapezoidal rule truncation error:} We apply Lemma~\ref{lem:mean-value-fractional-vector} for $f=u,a=t_j, \kappa=\tau, \mathrm{B}=\mathbb{L}^2(\Omega;\mathbb{H}_0^1(D))$ and $\gamma=1/2$ to get 
\[
\norm{\mathcal{Q}^{u}_{j+1/2}}_{\mathbb{L}^2(\Omega;\mathbb{H}_0^1(D))}\le \widetilde{C}\tau^{3/2},
\]
where $\widetilde{C}=\frac{4}{35}\|u_t\|_{C^{1/2}([0,T];\mathbb{L}^2(\Omega;\mathbb{H}_0^{1}(D)))}$, which is finite due the estimate~\eqref{eq:heat-regularity}. Finally, we get
\begin{align}\label{today09}
\tau \sum_{j=0}^{J-1} \mathbb{E}\bigl [\norm{\nabla\mathcal{Q}^{u}_{j+1/2}}_{\mathbb{L}^2(D)}^2\bigr ]\le C\tau^3,
\end{align}
where the constant depends on initial data $X_0$ and noise coefficients $\Phi$; see \eqref{eq:heat-regularity}.

\noindent
\textbf{Step 3. Micro-step moment bound:} In this step, we prove that the following bound holds
\begin{align}\label{today06}
\sum_{j=0}^{N-1} \mathbb{E}\bigl[\|\mathcal{Q}_j^W - \mathcal{I}_j^W\|^2_{\mathbb{R}^m}\bigr]\le C\tau^{4}.
\end{align}
Instead, for convenience, we define
\[
\mathcal{J}^W_j :=\mathcal{Q}_j^W - \mathcal{I}_j^W= \sum_{\ell=1}^{\mathrm{M}} \int_{t_{j,\ell-1}}^{t_{j,\ell}} \bigl(W(t_{j,\ell}) - W(s) \bigr) \, \mathrm{d}s,
\]
where $W$ is the $\mathbb{R}^m$-valued standard Wiener process with the covariance
\begin{align}\label{covariance}
    \mathbb{E}\bigl [(W(t) - W(s)) \otimes (W(t) - W(r))\bigr ] = \bigl( t - \max\{s, r\} \bigr) I, \qquad 0 \le s, r \le t,
\end{align}
where $I$ is the identity matrix in $\mathbb{R}^m$. By expanding the squared norm and using the independence of increments on disjoint micro-intervals we obtain
\[
\mathbb{E}\bigl [\|\mathcal{J}^W_j\|^2_{\mathbb{R}^m}\bigr ] = \sum_{\ell=1}^{\mathrm{M}} \int_{t_{j,\ell-1}}^{t_{j,\ell}} \int_{t_{j,\ell-1}}^{t_{j,\ell}} \mathbb{E}\Bigl [\bigl\langle W(t_{j,\ell}) - W(s), W(t_{j,\ell}) - W(r) \bigr\rangle_{\mathbb{R}^m}\Bigr ] \, \mathrm{d}s \, \mathrm{d}r.
\]
By applying the covariance identity \eqref{covariance} with $t = t_{j,\ell}$, we get
\begin{align}\label{today10}
\mathbb{E}\bigl [\|\mathcal{J}^W_j\|^2_{\mathbb{R}^m}\bigr ] = m \sum_{\ell=1}^{\mathrm{M}} \int_{t_{j,\ell-1}}^{t_{j,\ell}} \int_{t_{j,\ell-1}}^{t_{j,\ell}} \bigl( t_{j,\ell} - \max\{s, r\} \bigr) \, \mathrm{d}s \, \mathrm{d}r.
\end{align}
By fixing a micro-interval $[a, b] = [t_{j,\ell-1}, t_{j,\ell}]$ of length $\tau^2 = b - a$ we need to compute
\[
I_\ell := \int_a^b \int_a^b \bigl( b - \max\{s, r\} \bigr) \, \mathrm{d}s \, \mathrm{d}r.
\]
We split the integration domain into the triangles $\{r \ge s\}$ and $\{s > r\}$. By symmetry we obtain
\[
I_\ell = 2 \int_a^b \Bigl( \int_a^r (b - r) \, \mathrm{d}s \Bigr) \mathrm{d}r = 2 \int_a^b (r - a)(b - r) \, \mathrm{d}r.
\]
Thus we get
$
I_\ell = 2 \cdot \frac{(b - a)^3}{6} = \frac{\tau^6}{3}.
$ By returning to the main expression~\eqref{today10}, we obtain
\[
\mathbb{E}\bigl [\|\mathcal{J}^W_j\|^2_{\mathbb{R}^m}\bigr ] = m \sum_{\ell=1}^{\mathrm{M}} \frac{\tau^6}{3} = \frac{m}{3} \mathrm{M} \tau^6.
\]
Since $\mathrm{M} = \tau^{-1}$, we have
\[
\mathbb{E}\bigl [\|\mathcal{J}^W_j\|^2_{\mathbb{R}^m}\bigr ] = \frac{m}{3} \tau^{-1} \cdot \tau^6 = \frac{m}{3} \tau^5.
\]
By summing over $j$ from $0$ to $N-1$ with $N\tau = T$, we obtain
\[
\sum_{j=0}^{N-1} \mathbb{E}\bigl [\|\mathcal{J}^W_j\|^2_{\mathbb{R}^m}\bigr ] = N \cdot \frac{m}{3} \tau^5 = \frac{m}{3} T \tau^4.
\]
It gives the estimate~\eqref{today06}.

\noindent
\textbf{Step 4. Strong convergence rate:} By making use of estimates~\eqref{today09}-\eqref{today06} and $e_0 = 0$, from the error inequality~\eqref{today08} we obtain for all $1\le J\le N$,
\[
\mathbb{E}\bigl [\norm{e_J}_{\mathbb{L}^2(D)}^2\bigr ] + \tau \sum_{j=0}^{J-1} \mathbb{E}\bigl [\norm{\nabla (e_{j+1} + e_j)}_{\mathbb{L}^2(D)}^2\bigr ] \le C \tau^3.
\]
By taking square roots, we finally obtain the result of Theorem~\ref{thm:main}.
\end{proof}
\section{Stochastic Wave Equation}
\label{sec:wave}
Let \(D\subset\mathbb{R}^d\) be a smooth bounded domain. Let $(\Omega, \mathcal{F},\mathbb{F}, {\mathbb P})$ be a given probability space, and $W$ as specified in Section~\ref{section 2}. Assume $X_0\in \mathbb{H}_0^1(D)$, $X_1\in \mathbb{L}^2(D)$, $\Phi\in \bigl(\mathbb{L}^2(D)\bigr)^m$, and $A$ be from \eqref{operator} and uniformly elliptic. Then \eqref{eq:wave} admits a unique variational solution $X\in\mathbb{L}^2_{\mathbb F}\bigl(\Omega;C([0,T];\mathbb{H}_0^1(D))\bigr)$ such that $X_t\in \mathbb{L}^2_{\mathbb F}\bigl(\Omega;C([0,T];\mathbb{L}^2(D))\bigr),$
and the basic energy estimate 
\[
\mathbb{E}\bigl[\|X\|_{C([0,T];\mathbb{H}^1_0(D))}^2+\|X_t\|_{C([0,T];\mathbb{L}^2(D))}^2\bigr]
\le C\bigl(\|X_0\|_{\mathbb{H}_0^1(D)}^2+\|X_1\|_{\mathbb{L}^2(D)}^2+\|\Phi\|_{(\mathbb{L}^2(D))^m}^2\bigr)
\]
holds. Moreover, for more regular data 
$
X_0\in \mathbb{H}^2(D)\cap \mathbb{H}_0^1(D), X_1\in \mathbb{H}_0^1(D),
$ and $\Phi\in\bigl(\mathbb{H}_0^1(D)\bigr)^m$, we have
$
X\in \mathbb{L}^2_{\mathbb F}\big(\Omega;C([0,T];\mathbb{H}^2(D)\cap \mathbb{H}_0^1(D))\big),
X_t\in \mathbb{L}^2_{\mathbb F}\big(\Omega;C([0,T];\mathbb{H}_0^1(D))\big),$ and 
\[
\mathbb{E}\bigl[\|X\|_{C([0,T];\mathbb{H}^2(D))}^2+\|X_t\|_{C([0,T];\mathbb{H}_0^1(D))}^2\bigr]
\le C\bigl(\|X_0\|_{\mathbb{H}^2(D)}^2+\|X_1\|_{\mathbb{H}_0^1(D)}^2+\|\Phi\|_{(\mathbb{H}_0^1(D))^m}^2\bigr)\,;
\]
see, {\em e.g.,} \cite[Sec. 5.3]{Chow2014}.
For {\em random} PDE~\eqref{eq:wave-transformed}, if $X_0\in\mathbb{H}_0^1(D), X_1\in\mathbb{L}^2(D)$ and $\Phi\in \bigl(\mathbb{H}^2(D)\cap H_0^1(D)\bigr)^m$, then there exists the unique variational solution $u\in\mathbb{L}^2_{\mathbb{F}}(\Omega; C([0,T];\mathbb{H}_0^1(D)))$ to {\em random} PDE~\eqref{eq:wave-transformed} such that $u_t\in\mathbb{L}^2_{\mathbb{F}}(\Omega;C([0,T];\mathbb{L}^2(D)))$ , $u_{tt}\in \mathbb{L}^2_{\mathbb{F}}(\Omega;\mathbb{L}^2([0,T];\mathbb{H}^{-1}(D)))$, and the following variational formulation holds: $\mathbb{P}$-a.s., for all $\varphi\in \mathbb{H}_0^1(D)$, a.e. $s\in [0,T]$,
\begin{align*}
    \left< u_{tt}(s), \varphi\right>_{\mathcal{D}}=-\left<\nabla u(s), \nabla \varphi\right>+ \left<\int_0^s\Delta\bigl[\Phi W(\tau)\bigr]\,{\rm d}\tau,\varphi\right>,
\end{align*}
with $u(0)=X_0$ and $u_t(0)=X_1$; see, {\em e.g.,} \cite[Ch. 7, Sec. 2]{Evans2010}. Moreover, for the more regular data $X_0\in\mathbb{H}^4(D)\cap \mathbb{H}_0^1(D)$, $X_1\in\mathbb{H}^3(D)\cap\mathbb{H}_0^1(D), \Phi \in \bigl(\mathbb{H}^4(D)\cap \mathbb{H}_0^1(D)\bigr)^m$, with compatibility conditions $AX_0, AX_1\in \mathbb{H}_0^1(D)$, and $ A\Phi\in \bigl(\mathbb{H}_0^1(D)\bigr)^m$, with the help of the path-wise H\"older continuity of $A[\Phi W]$ in time \big({\em i.e.,} $A [\Phi W]\in \mathbb{L}^2(\Omega;C([0,T];\mathbb{H}_0^1(D)))$\big) we have the following regularity in time 
\[
u\in \mathbb{L}^2_{\mathbb F}\big(\Omega;C^2([0,T];\mathbb{H}^2(D))\big),\quad u_t\in \mathbb{L}^2_{\mathbb F}\big(\Omega;C^2([0,T];\mathbb{L}^2(D))\big).
\]
with the following bound
\begin{equation}\label{eq: wave-regularity}
\mathbb{E}\bigl[\|u\|_{C^2([0,T];\mathbb{H}^2(D))}^2+\|u_t\|_{C^2([0,T];\mathbb{L}^2(D)))}^2\bigr]\le C\bigl(\|X_0\|_{\mathbb{H}^4(D)}^2+\|X_1\|_{\mathbb{H}^3(D)}^2+\|\Phi\|_{(\mathbb{H}^4(D))^m}^2\bigr).\,
\end{equation}
This regularity bound may {\em e.g.} be deduced from \cite[Ch. 7, Sec. 2]{Evans2010} and \cite[Thm. 5.3.1, Thm. 5.3.2]{Pazy1983} by exploiting temporal H\"older regularity of the Wiener process on the right-hand side of \eqref{eq:wave-transformed} \footnote{Here we apply \cite[Ch. 7, Sec. 2]{Evans2010} and \cite[Thm.~5.3.1, Thm.~5.3.2]{Pazy1983} pathwise ({\em i.e.},\ $\mathbb{P}$-almost surely), and subsequently take expectation to obtain the estimate \eqref{eq: wave-regularity}.}.
\subsection{Modified Crank--Nicolson scheme for SPDE~\eqref{eq:wave}}
\label{subsec:wave-modified}
For simplicity, in this section, we take $A = -\Delta$. Our main result (Theorem~\ref{thm:wave-main-precise} below), however, remains valid for any $A$ from \eqref{operator} as long as it is uniformly elliptic with smooth coefficients. We now present a time discretization for the stochastic wave equation~\eqref{eq:wave}. Let again $0=t_0<t_1<\dots<t_N=T$ be a uniform mesh grid with step $\tau=T/N$. We follow the same basic idea used for the parabolic problem in Section~\ref{sec:heat}. By integrating the second-order equation \eqref{eq:wave-transformed} with $A=-\Delta$ on the sub-interval $[t_j,t_{j+1}]$, we get the following integral equations
\begin{equation}\label{eq:wave-integral}
\begin{aligned}
u(t_{j+1})-u(t_j) &= \int_{t_j}^{t_{j+1}}  u_t(s)\,\mathrm{d}s,\\[4pt]
u_t(t_{j+1})-u_t(t_j) &= \int_{t_j}^{t_{j+1}} \Delta u(s)\,\mathrm{d}s
+ \int_{t_j}^{t_{j+1}}\int_0^t \Delta\bigl[\Phi W(s)\bigr]\,\mathrm{d}s\,\mathrm{d}t.
\end{aligned}
\end{equation}
The three integrals on the right-hand side require a separate treatment:
the deterministic integrals $\int_{t_j}^{t_{j+1}} u_t(s)\,\mathrm{d}s,$ and $\displaystyle\int_{t_j}^{t_{j+1}} \Delta u(s)\,\mathrm{d}s$ possess Lipschitz integrands in time and so may benefit from the trapezoidal / Crank--Nicolson rule; the double integral
$
\int_{t_j}^{t_{j+1}}\Big(\int_0^t \bigl[\Delta\Phi\big]\,W(s)\,\mathrm{d}s\Big)\mathrm{d}t
$
has a non-Lipschitz integrand and will therefore be approximated by a Riemann sum on the finer mesh of sizw.

\medskip

\noindent\textbf{1. Approximation of the deterministic integral.} By trapezoidal (Crank--Nicolson) rule, we approximate$$\mathcal{Q}_{j}^{u_t}:=\int_{t_j}^{t_{j+1}}u_t(s)\,{\rm d}s\qquad\text{and}\qquad\mathcal{Q}_j^\Delta := \int_{t_j}^{t_{j+1}} \Delta u(s)\,\mathrm{d}s$$ in \eqref{eq:wave-integral} by
$$
\mathcal{I}_j^{u}:= \frac{\tau}{2}\,\bigl[u_t(t_{j+1})+u_t(t_j)\bigr]\qquad\text{and}\qquad\mathcal{I}_j^\Delta := \frac{\tau}{2}\,\Delta\bigl[u(t_{j+1})+u(t_j)\bigr],
$$
respectively.
This approximation is second order in $\tau$ provided $u$ has proper time regularity.

\medskip

\noindent\textbf{2. Approximation of the double integral with stochastic integrand.}  On each sub-interval $[t_j,t_{j+1}]$ we define the finer mesh nodes $$t_{j,\ell}=t_j+\ell\tau^2,\qquad \ell=0,\dots,M=\tau^{-1}.$$  
By Fubini's theorem on domain $\{(t,s): t_j\le t\le t_{j+1},\, 0\le s\le t\}$, we split the double integral in two parts as
\begin{align*}
\mathcal{Q}_{j+1}^W:=\int_{t_j}^{t_{j+1}}\int_{0}^{t} W(s)\,{\rm d}s\,{\rm d}t& = \int_{0}^{t_{j+1}}\left(\int_{\max\{t_j,s\}}^{t_{j+1}}\,{\rm d}t\right) W(s)\,{\rm d}s\\&= \int_{0}^{t_{j+1}}\bigl(t_{j+1}-\max\{t_j,s\}\bigr)\,W(s)\,{\rm d}s \\
&= \int_{0}^{t_j}\bigl(t_{j+1}-t_j\bigr) W(s)\,{\rm d}s
\;+\; \int_{t_j}^{t_{j+1}}\bigl(t_{j+1}-s\bigr) W(s)\,{\rm d}s \\&= \tau{\int_{0}^{t_j}W(s)\,{\rm d}s}
\;+\;\int_{t_j}^{t_{j+1}}(t_{j+1}-s)\,W(s)\,{\rm d}s
\end{align*}
By using the Riemann sum on the finer mesh we approximate $\mathcal{Q}_{j+1}^W$ by
\begin{equation}\label{eq:Jtilde}
\mathcal{I}^W_{j+1}
:= \sum_{m=0}^{j-1}\sum_{\ell=1}^{\mathrm{M}} \tau^3\,\,W(t_{m,\ell})
\;+\; \sum_{\ell=1}^{\mathrm{M}} \big(t_{j+1}-t_{j,\ell}\big)\,\tau^2\,W(t_{j,\ell}).
\end{equation}
Let $(U_j,V_j)$ denote approximations to $(u(t_j), u_t(t_j))$.  Motivated by \eqref{eq:wave-integral} and the above quadratures we employ the following implicit Crank--Nicolson for {\em random} PDE~\eqref{eq:wave},
\begin{equation}\label{eq:wave-scheme-random}
\begin{cases}
U_{j+1}-U_j = \dfrac{\tau}{2}\,\bigl[V_{j+1}+V_j\bigr],\\[8pt]
V_{j+1}-V_j = \frac{\tau}{2}\,\Delta\,[U_{j+1}+U_j] \;+\sum_{m=0}^{j-1}\sum_{\ell=1}^{\mathrm{M}} \tau^3\,\Delta\bigl[\Phi\,W(t_{m,\ell})\bigr]
\;+\; \sum_{\ell=1}^{\mathrm{M}} \big(t_{j+1}-t_{j,\ell}\big)\,\tau^2\,\Delta\bigl[\Phi\,W(t_{j,\ell})\bigr],
\end{cases}
\end{equation}
$j=0,\dots,N-1,$ with $U_0=X_0$, $V_0=X_1$ and homogeneous Dirichlet boundary conditions on $U_j$ for all $j$.

\noindent
As in the stochastic heat equation case in Section \ref{sec:heat}, we reconstruct approximations for SPDE~\eqref{eq:wave-transformed} by reverting transformation~\eqref{today11} on a discrete level,
\begin{align}\label{today10-1}
X_j := U_j + \sum_{m=0}^{j-1}\sum_{\ell=1}^{\mathrm{M}}\tau^2\,\Phi W(t_{m,\ell}),\qquad Y_j := V_j + \Phi W(t_j).
\end{align}
By inserting these into \eqref{eq:wave-scheme-random} we herewith arrive a new {\em modified Crank-Nicolson scheme} to compute $\{(X_{j},Y_j);\,0\le j\le N\}$ to approximate $(X, X_t)$ with solution $X$ to SPDE~\eqref{eq:wave} --- which is summarized in the following scheme.
\medskip
\begin{algorithm}[H]
\caption{modified {\em Crank--Nicolson scheme} for SPDE~\eqref{eq:wave} with $A=-\Delta$}
\label{alg:wave-modified-detailed}
\begin{algorithmic}[1]
\State \textbf{Inputs:} domain $D$, operator $A=-\Delta$, noise coefficient $\Phi$, final time $T$, coarse steps $N$, $\tau=T/N$, initial data $X_0,X_1$, Wiener process $W$.
\State \textbf{Coarse grid:} $t_j=j\tau,\ j=0,\dots,N$.
\State \textbf{Microgrid:} $t_{j,\ell}=t_j+\ell\tau^2$, $\ell=0,\dots,\mathrm{M}$, $\mathrm{M}=\tau^{-1}$.
\State \textbf{Initialization:} set $X_0, X_1\in\mathbb{H}_0^1(D)$.
\For{$j=0$ \textbf{to} $N-1$}
\State Compute the Wiener increment: $\Delta_{j+1}W = W(t_{j+1}) - W(t_j)$.
    \State Compute correction terms:
\begin{align*}
\mathfrak{E}_{1,j}(W)
&:= \; \sum_{\ell=1}^{\rm M}\tau^2\,\Phi W(t_{j,\ell}) \;-\; \frac{\tau}{2}\,\Phi\bigl[W(t_{j+1})+W(t_j)\bigr],\\[6pt]
\label{eq:E2-wave}
\mathfrak{E}_{2,j}^\Delta(W)
&:= \frac{1}{2}\sum_{\ell=1}^{\mathrm{M}} \big(2t_{j+1}-\tau-2t_{j,\ell}\big)\,\tau^2\,\Delta\bigl[\Phi W(t_{j,\ell})\bigr].
\end{align*}
\afterpage{\clearpage}
\State Solve the following {\em modified Crank-Nicolson scheme} for $\bigl(X_{j+1},Y_{j+1}\bigr)\in \mathbb{H}_0^1(D)\times\mathbb{H}_0^{1}(D)$:
\begin{align*}
\begin{cases}
X_{j+1}-X_j
&= \dfrac{\tau}{2}\,\bigl[Y_{j+1}+Y_j\bigr] \;+\; \mathfrak{E}_{1,j}(W),\\[8pt]
Y_{j+1}-Y_j
&= \dfrac{\tau}{2}\,\Delta\,\big[X_{j+1}+X_j\big] \;+\; \Phi\Delta_{j+1}W
\;+\; \mathfrak{E}_{2,j}^\Delta(W).
\end{cases}
\end{align*}
\EndFor
\State \textbf{Output:} $\{(X_j,Y_j);\,0\le j\le N\}$.
\end{algorithmic}
\end{algorithm}
Note that the correction terms $\mathfrak{E}_{1,j}( W)$ and $\mathfrak{E}_{2,j}^\Delta(W)$ in the {\em modified Crank--Nicolson} Scheme~\ref{alg:wave-modified-detailed} play a key role in improving the overall convergence rate. The following theorem settles ${\mathcal O}(\tau^{2})$ convergence for the {\em modified Crank--Nicolson} Scheme~\eqref{alg:wave-modified-detailed}.
\medskip
\begin{theorem}[second main result]\label{thm:wave-main-precise}
Let $X_0\in\mathbb{H}^4(D)\cap \mathbb{H}_0^1(D)$, $X_1\in\mathbb{H}^3(D)\cap\mathbb{H}_0^1(D), \Phi \in \bigl(\mathbb{H}^4(D)\cap \mathbb{H}_0^1(D)\bigr)^m$, with compatibility conditions $AX_0, AX_1\in \mathbb{H}_0^1(D)$, and $ A\Phi\in \bigl(\mathbb{H}_0^1(D)\bigr)^m$ hold. Let $\{(X_j,Y_j)\}_{j=0}^N$ be generated by the {\em modified Crank-Nicolson} Scheme~\ref{alg:wave-modified-detailed} and $X$ solves SPDE~\eqref{eq:wave} with $A=-\Delta$. Then there exists a constant $C=C(T, X_0, X_1, \Phi, m)>0$, independent of $\tau$, such that

\begin{equation}\label{today16}
\max_{0\le j\le N}
\Big\{
\E\big[\|X(t_j)-X_j\|_{\mathbb{H}^1_0(D)}^2\big]^{1/2}
\;+\;
\E\big[\|X_t(t_j)-Y_j\|_{\mathbb{L}^2(D)}^2\big]^{1/2}
\Big\}
\le C\,\tau^{2}.
\end{equation}

\end{theorem}
\begin{proof} We prove this result with the help of the modified Crank--Nicolson scheme~\eqref{eq:wave-scheme-random}  applied to the \emph{random} PDE~\eqref{eq:wave-transformed}. The proof proceeds through the following steps

\noindent
\textbf{Step 1. Error inequality:}
We define the errors
\[
e_{1,j}= u(t_j)-U_j,\qquad e_{2,j} = u_t(t_j)-V_j.
\]
Therefore from \eqref{eq:wave-integral} and \eqref{eq:wave-scheme-random} we get the following discrete error system
\begin{subequations}\label{eq:error-system-mod}
\begin{align}
e_{1, j+1}-e_{1,j} &= \frac{\tau}{2}\bigl[e_{2,j+1}+e_{2, j}\bigr] + \tau\mathcal{Q}^{u_t}_{j+1/2}, \label{eq:error-U-mod}\\
{e_{2,j+1}-e_{2,j}} &= \frac{\tau}{2}\Delta\bigl[e_{1,j+1}+e_{1,j}\bigr] + \tau\Delta\mathcal{Q}^u_{j+1/2} + \Delta\bigl[\Phi(\mathcal{Q}_{j+1}^W-\mathcal{I}_{j+1}^W)\bigr],\label{eq:error-V-mod}
\end{align}
\end{subequations}
where $$\mathcal{Q}^u_{j+1/2}=\frac{1}{\tau}\int_{t_{j}}^{t_{j+1}}u(s)\,{\rm d}s-\frac{1}{2}\bigl[u(t_j)+u(t_{j+1})\bigr],$$
$$\mathcal{Q}^{u_t}_{j+1/2}=\frac{1}{\tau}\int_{t_{j}}^{t_{j+1}}u_t(s)\,{\rm d}s-\frac{1}{2}\bigl[u_t(t_j)+u_t(t_{j+1})\bigr].$$
By taking the $\mathbb{L}^2(D)$-inner product of \eqref{eq:error-U-mod} with $-\Delta e_{1,j+1/2},$ where $ e_{1,j+1/2}:=\frac{1}{2}{\bigl[e_{1, j+1}+e_{1,j}\bigr]}$ and \eqref{eq:error-V-mod} with $e_{2,j+1/2}:=\frac{1}{2}\bigl[e_{2,j+1}+e_{2,j}\bigr]$, then we add the both identities to get the error identity
\begin{align*}
&\frac{1}{2}\bigl(\|\nabla e_{1,j}\|_{\mathbb{L}^2(D)}^2-\|\nabla e_{1,j}\|_{\mathbb{L}^2(D)}^2\bigr)+\frac{1}{2}\bigl(\|e_{2,j}\|_{\mathbb{L}^2(D)}^2-\|e_{2,j}\|_{\mathbb{L}^2(D)}^2\bigr)\\&={\tau}\left<\nabla\mathcal{Q}_{j+1/2}^{u_t},\nabla e_{1,j+1/2}\right>+\tau\left<\Delta\mathcal{Q}_{j+1/2}^u, e_{2,j+1/2}\right>-\tau\left<\frac{1}{\tau}\nabla\bigl[\Phi\bigl(\mathcal{Q}_{j+1/2}^W-\mathcal{I}_{j+1}^W\bigr)\bigr], \nabla e_{2,j+1/2}\right>.
\end{align*}
For any $1\le J\le N$, by applying Young's inequality, taking expectation, and summing over all steps from $0$ to $J-1$, we get the error inequality
\begin{align}\label{today08-1}
    \E\bigl[\|\nabla e_{1,J}\|_{\mathbb{L}^2(D)}^2\bigr]+\E\bigl[\|e_{2,J}\|_{\mathbb{L}^2(D)}^2\bigr]&\le\E\bigl[\|\nabla e_{1,0}\|_{\mathbb{L}^2(D)}^2\bigr]+\E\bigl[\|e_{2,0}\|_{\mathbb{L}^2(D)}^2\bigr]\notag\\&\qquad+C\,\tau\sum_{j=0}^{J-1}\bigl(\E\bigl[\|\nabla e_{1,j}\|_{\mathbb{L}^2(D)}^2\bigr]+\E\bigl[\|e_{2,j}\|_{\mathbb{L}^2(D)}^2\bigr]\bigr)\notag\\&\qquad+\tau \sum_{j=0}^{J-1}\E\bigl[\|\Delta\mathcal{Q}_{j+1/2}^u\|_{\mathbb{L}^2(D)}^2\bigr]+\tau \sum_{j=0}^{J-1}\E\bigl[\|\nabla\mathcal{Q}_{j+1/2}^{u_t}\|_{\mathbb{L}^2(D)}^2\bigr]\notag\\&\qquad+\frac{1}{\tau}\sum_{j=0}^{J-1}\E\bigl[\bigl\|\nabla\bigl[\Phi\bigl(\mathcal{Q}_{j+1}^W-\mathcal{I}_{j+1}^W\bigr)\bigr]\bigr\|_{\mathbb{L}^2(D)}^2\bigr].
\end{align}

\noindent
\textbf{Step 2. Trapezoidal rule truncation error:} We apply Lemma~\ref{lem:mean-value-fractional-vector} for $f=u,a=t_j, \kappa=\tau, \mathrm{B}=\mathbb{L}^2(\Omega;\mathbb{H}^2(D))$ and $\gamma=1$ to get 
\[
\norm{\mathcal{Q}^{u}_{j+1/2}}_{\mathbb{L}^2(\Omega;\mathbb{H}^2(D))}\le \widetilde{C}\tau^{2},
\]
where $\widetilde{C}=\frac{4}{35}\|u\|_{C^{2}([0,T];\mathbb{L}^2(\Omega;\mathbb{H}^{2}(D)))}$ which is finite due the estimate~\eqref{eq: wave-regularity}. Finally, we get
\begin{align}\label{today09-1}
\tau \sum_{j=0}^{J-1} \mathbb{E}\bigl [\norm{\Delta\mathcal{Q}^{u}_{j+1/2}}_{\mathbb{L}^2(D)}^2\bigr ]\le C\tau^4,
\end{align}
where the constant depends on the initial data $X_0$ and the noise coefficients $\Phi$; see \eqref{eq: wave-regularity}. Again we apply Lemma~\ref{lem:mean-value-fractional-vector} for $f=u_t,a=t_j, \kappa=\tau, \mathrm{B}=\mathbb{L}^2(\Omega;\mathbb{H}^1(D))$ and $\gamma=1$ to get 
\[
\norm{\mathcal{Q}^{u_t}_{j+1/2}}_{\mathbb{L}^2(\Omega;\mathbb{H}^1(D))}\le \widetilde{C}\tau^{2},
\]
where $\widetilde{C}=\frac{4}{35}\|u_t\|_{C^{2}([0,T];\mathbb{L}^2(\Omega;\mathbb{H}^1(D))})$ which is finite due the estimate~\eqref{eq: wave-regularity}. Finally, we get
\begin{align}\label{today09-2}
\tau \sum_{j=0}^{J-1} \mathbb{E}\bigl [\norm{\nabla\mathcal{Q}^{u}_{j+1/2}}_{\mathbb{L}^2(D)}^2\bigr ]\le C\tau^4,
\end{align}
where the constant depends on the initial data $X_0$ and the noise coefficients $\Phi$; see \eqref{eq: wave-regularity}.

\noindent
\textbf{Step 3. Micro-step error:}
In this step, we prove the following estimate
\begin{equation}
\label{eq:double-err-corrected-local}
\sum_{j=0}^{J}\mathbb{E}\bigl[\|\mathcal{Q}_{j+1}^W - \mathcal{I}_{j+1}^W\|_{\mathbb{R}^m}^2\bigr] \le C \tau^5.
\end{equation}
For convenience we define
\[
\mathcal{J}_j^W := \mathcal{Q}_{j+1}^W - \mathcal{I}_{j+1}^W = \mathcal{J}_j^{\mathrm{old}}+\mathcal{J}_j^{\mathrm{local}},
\]
where
\begin{align}\label{today12}
\mathcal{J}_j^{\mathrm{old}} &:= \tau\int_0^{t_j} W(s)\,{\rm d}s \;-\;\tau\sum_{m=1}^{j-1}\sum_{\ell=1}^{\mathrm{M}} \tau^2\,W(t_{m,\ell})
= \tau\sum_{m=1}^{j-1}\sum_{\ell=1}^{\mathrm{M}}\int_{t_{m-1,\ell}}^{t_{m,\ell}}\bigl( W(s)-W(t_{m,\ell})\bigr)\,{\rm d}s,\\[4pt]
\mathcal{J}_j^{\mathrm{local}}
&:= \int_{t_j}^{t_{j+1}}(t_{j+1}-s)\,W(s)\,{\rm d}s
- \sum_{\ell=1}^{\mathrm{M}}(t_{j+1}-t_{j,\ell})\,\tau^2\,W(t_{j,\ell})\notag\\[4pt]
&= \sum_{\ell=1}^{\mathrm{M}}\Bigg(\int_{t_{j,\ell-1}}^{t_{j,\ell}}(t_{j+1}-s)\,W(s)\,{\rm d}s
-(t_{j+1}-t_{j,\ell})\,W(t_{j,\ell})\Bigg).
\end{align}
For further simplification of the term $\mathcal{J}_j^{\mathrm{local}}:=\sum_{\ell=1}^{\mathrm{M}}\mathcal{J}_{j,\ell}^{\mathrm{local}}$, we split \(W(s)=W(t_{j,\ell})+(W(s)-W(t_{j,\ell}))\) inside the integral to obtain
\[
\begin{aligned}
\mathcal{J}_{j,\ell}^{\mathrm{local}}
&= \int_{t_{j,\ell-1}}^{t_{j,\ell}}(t_{j+1}-s)\bigl(W(s)-W(t_{j,\ell})\bigr)\,{\rm d}s
+ W(t_{j,\ell})\int_{t_{j,\ell-1}}^{t_{j,\ell}}\bigl((t_{j+1}-s)-(t_{j+1}-t_{j,\ell})\bigr)\,{\rm d}s\\[4pt]
&= \int_{t_{j,\ell-1}}^{t_{j,\ell}}(t_{j+1}-s)\bigl(W(s)-W(t_{j,\ell})\bigr)\,{\rm d}s
+ W(t_{j,\ell})\int_{t_{j,\ell-1}}^{t_{j,\ell}}(t_{j,\ell}-s)\,{\rm d}s\\[4pt]
&= \int_{t_{j,\ell-1}}^{t_{j,\ell}}(t_{j+1}-s)\bigl(W(s)-W(t_{j,\ell})\bigr)\,{\rm d}s
+ \frac{\tau^4}{2}\,W(t_{j,\ell}).
\end{aligned}
\]
Then after taking sum over $\ell$ from $0$ to $\mathrm{M}$, it gives $$\mathcal{J}_{j}^{\mathrm{local}}:=\mathcal{J}_{j}^{\mathrm{curr}}+\mathcal{S}^W_{j},$$
where 
\begin{align*}
\mathcal{J}_{j}^{\mathrm{curr}}
:&=\sum_{\ell=1}^{\mathrm{M}}\int_{t_{j,\ell-1}}^{t_{j,\ell}}(t_{j+1}-s)\bigl(W(s)-W(t_{j,\ell})\bigr)\,{\rm d}s,\\[4pt]
\mathcal{S}^W_{j}&:=\sum_{\ell=1}^{\mathrm{M}}\frac{\tau^4}{2}\, W(t_{j,\ell}).
\end{align*}
Hence the exact decomposition is
\begin{equation}\label{eq:Ejcurr-decomp}
\mathcal{J}_j^W=\mathcal{J}_j^{\mathrm{old}}+\mathcal{J}_j^{\mathrm{curr}}+\mathcal{S}_{j}^W.
\end{equation}

\noindent
\textbf{Step 3.1. Estimate of $\mathcal{S}_{j}^W$:}
We now estimate the discrete term $\mathcal{S}_j^W$. We obtain
\[
\begin{aligned}
\E\bigl[\|\mathcal{S}^W_j\|_{\mathbb{R}^m}^2\bigr]
&= \frac{\tau^8}{4}
\E\bigg[\Bigl\|\sum_{\ell=1}^{\mathrm{M}}W(t_{j,\ell})\Bigr\|_{\mathbb{R}^m}^2\bigg]
= \frac{\tau^8}{4} \sum_{\ell,\ell'=1}^{\mathrm{M}}\E\bigl\langle W(t_{j,\ell}),W(t_{j,\ell'})\bigr\rangle_{\mathbb{R}^m
}\\[4pt]
&= \frac{m\tau^8}{4} \sum_{\ell,\ell'=1}^{\mathrm{M}}\min\{t_{j,\ell},t_{j,\ell'}\}.
\end{aligned}
\]
Since \(t_{j,p},t_{j,\ell'}\in[t_j,t_{j+1}]\) and \(\min\{t_{j,\ell},t_{j,\ell'}\}\le t_{j+1}\le T\), we obtain
\[
\sum_{\ell, \ell'=1}^{\mathrm{M}}\min\{t_{j,\ell},t_{j,\ell'}\}\le \mathrm{M}^2\,T.
\]
By using \(\mathrm{M}=\tau^{-1}\), this gives
\begin{align}\label{today01,w}
\E\bigl[\|\mathcal{S}^W_j\|_{\mathbb{R}^m}^2\bigr] \le m\tau^8\tau^{-2}T=  mT\tau^6.
\end{align}

\medskip\noindent
\textbf{Step 3.2. Estimate of $\mathcal{J}_{j}^{\mathrm{old}}$:}
We fix one micro interval $[t_{m-1,\ell},t_{m,\ell}]=[a,b]$ of length $\tau^2$ and set
\[
I_m := \int_{a}^{b}\bigl(W(s)-W({b})\bigr)\,{\rm d}s.
\]
For $s,r\in[a,b]$, the covariance identity gives
\[
\mathbb{E}\bigl [\bigl\langle W(s)-W(b),\,W(r)-W(b)\bigr\rangle\bigr ]
= \bigl(b-\max\{s,r\}\bigr)\,{I}.
\]
Hence
\[
\begin{aligned}
\mathbb{E}\bigl[\|I_m\|_{\mathbb{R}^m}^2\bigr]
&= m\int_{a}^{b}\int_{a}^{b} \bigl(b-\max\{s,r\}\bigr)\,{\rm d}s\,{\rm d}r= m\int_0^{\tau^2}\int_0^{\tau^2}\min\{x,y\}\,{\rm d}x\,{\rm d}y\le \frac{\tau^6m}{3},
\end{aligned}
\]
where in the second line we used the change of variables $x=b-s,\;y=b-r$. By summing over the disjoint micro intervals before $t_j$ and by independence of Wiener increments on micro intervals, we yield
\[
\mathbb{E}\bigl[\|\mathcal{J}_j^{\mathrm{old}}\|_{\mathbb{R}^m}^2\bigr]
= \tau^2\sum_{m=1}^{j-1}\sum_{\ell=1}^{\mathrm{M}}\mathbb{E}\bigl[\|I_{m,\ell}\|_{\mathbb{R}^m}^2\bigr]
\le \frac{(t_j/\tau) m\,\mathrm{M},\tau^8}{3}
=\frac{t_j m\tau^5}{3}\le \frac{T m\tau^5}{3},
\]
where the facts $j-1\le {t_j}/\tau$, $t_j\le T$ and $\mathrm{M}=\tau^{-1}$ are used.
Finally we obtain
\begin{equation}\label{est:old}
\mathbb{E}\bigl[\|\mathcal{J}_j^{\mathrm{old}}\|_{\mathbb{R}^m}^2\bigr] \le C\tau^{6}.
\end{equation}
\medskip\noindent
\textbf{Step 3.3. Estimate of $\mathcal{J}_j^{\mathrm{curr}}$:} For a micro interval $[t_{j,\ell-1},t_{j,\ell}]\subset[t_j,t_{j+1}]$, we define
\[
I_{j,\ell} := \int_{t_{j,\ell-1}}^{t_{j,\ell}}(t_{j+1}-s)\bigl(W(s)-W(t_{j,\ell})\bigr)\,{\rm d}s.
\]
Then, similarly as previous in the sub-step, we compute
\[
\mathbb{E}\bigl [\|I_{j,\ell}\|_{\mathbb{R}^m}^2\bigr ]
= m\int_{t_{j,\ell-1}}^{t_{j,\ell}}\int_{t_{j,\ell-1}}^{t_{j,\ell}}
(t_{j+1}-s)(t_{j+1}-r)\bigl(t_{j,\ell}-\max\{s,r\}\bigr)\,{\rm d}s\,{\rm d}r.
\]
For $s,r\in[t_{j,\ell-1},t_{j,\ell}]$ we have $0\le t_{j+1}-s\le \tau$ and $0\le t_{j,\ell}-\max\{s,r\}\le \tau^2$. Therefore we get
\[
\mathbb{E}\bigl [\|I_{j,\ell}\|_{\mathbb{R}^m}^2\bigr ] \le m\tau^4\int_{t_{j,\ell-1}}^{t_{j,\ell}}\int_{t_{j,\ell-1}}^{t_{j,\ell}} \,{\rm d}s\,{\rm d}r
= m\tau^8.
\]
Hence
\[
\begin{aligned}
\mathbb{E}\bigl[\|\mathcal{J}_j^{\mathrm{curr}}\|_{\mathbb{R}^m}^2\bigr]
&= \sum_{\ell=1}^{\mathrm{M}}\mathbb{E}\bigl[\|I_{j,\ell}\|_{\mathbb{R}^m}^2\bigr]
\le m\mathrm{M}\;\tau^7.
\end{aligned}
\]
By using $\mathrm{M}=\tau^{-1}$, we obtain
\begin{equation}\label{est:curr}
\mathbb{E}\bigl[\|\mathcal{J}_j^{\mathrm{curr}}\|_{\mathbb{R}^m}^2\bigr] \le m\tau^{6}.
\end{equation}
Finally, we combine \eqref{today01,w}-- \eqref{est:curr} to get \eqref{eq:double-err-corrected-local}.

\medskip\noindent
\textbf{Step 4. Rate of convergence for {\em random} PDE~\eqref{eq:wave-transformed}:} By making use of the estimates \eqref{today09-1}-\eqref{eq:double-err-corrected-local} and $e_{1,0}=e_{2,0}=0$, from the energy inequality~\eqref{today08-1}, we get
\begin{align}\label{today08-11}
    \E\bigl[\|\nabla e_{1,J}\|_{\mathbb{L}^2(D)}^2\bigr]&+\E\bigl[\|e_{2,J}\|_{\mathbb{L}^2(D)}^2\bigr]\le\tau\sum_{j=0}^{J-1}\bigl(\E\bigl[\|\nabla e_{1,j}\|_{\mathbb{L}^2(D)}^2\bigr]+\E\bigl[\|e_{2,j}\|_{\mathbb{L}^2(D)}^2\bigr]\bigr)+C\tau^4.
\end{align}
We now apply discrete Gronwall's inequality to get
\begin{align}\label{today13}
\max_{0\le j\le N}
\Bigg\{
\E\big[\|\nabla\big(u(t_j)-U_j\big)\|_{\mathbb{L}^2(D)}^2\big]^{1/2}
\;+\;
\E\big[\|u_t(t_j)-V_j\|_{\mathbb{L}^2(D)}^2\big]^{1/2}
\Bigg\}
\le C\,\tau^{2}.
\end{align}
\textbf{Step 5. Rate of convergence for SPDE~\eqref{eq:wave}:}
To get the estimate~\eqref{today16}, we use the transformations \eqref{today11},~\eqref{today10-1}, and identity~\eqref{today12} to conclude that
\begin{align}
  X(t_j)-X_j&=u(t_j)-U_j + \int_0^{t_j}\Phi W(s)\,{\rm d}s-\sum_{m=1}^{j-1}\sum_{\ell=1}^{\mathrm{M}}\tau^2\Phi W(t_{m,\ell})=u(t_j)-U_j+\frac{1}{\tau}\Phi\mathcal{J}_j^{\mathrm{old}},\\ X_t(t_j)-Y_j&=u_t(t_j)-V_j.
\end{align}
It implies that 
\begin{align*}
\mathbb{E}\bigl[\|\nabla \bigl(X(t_j)-X_j\bigr)\|_{\mathbb{L}^2(D)}^2\bigr]\le\E\big[\|\nabla\big(u(t_j)-U_j\big)\|_{\mathbb{L}^2(D)}^2\big] +\frac{1}{\tau^2}\mathbb{E}\bigl[\|\mathcal{J}_j^{\mathrm{old}}\|_{\mathbb{R}^m}^2\bigr]\|\nabla\Phi\|_{(\mathbb{L}^2(D))^m}^2.
\end{align*}
By using the estimates~\eqref{est:old} and \eqref{today13}, we conclude that there exists $C>0$, independent of $j$, such that
\begin{align}\label{today14}
   \mathbb{E}\bigl[\|\nabla \bigl(X(t_j)-X_j\bigr)\|_{\mathbb{L}^2(D)}^2\bigr]\le C\tau^4.
\end{align}
Finally we combine \eqref{today13}--\eqref{today14} to get the convergence rate estimate~\eqref{today16} of this theorem.
\end{proof}

\end{document}